\documentclass[12pt,letterpaper]{amsart}
\usepackage{geometry}
\usepackage{amsmath,amssymb,amsfonts,amsthm,amscd}
\usepackage{mathrsfs,latexsym,url,graphicx,setspace}
\usepackage{color}
\newtheorem{theorem}{Theorem}
\newtheorem{proposition}{Proposition}
\newtheorem{lemma}{Lemma}
\newtheorem{corollary}{Corollary}
\theoremstyle{remark}
\newtheorem{question}{Question}
\newtheorem{remark}{Remark}

\newcommand{\C}{\mathbb{C}}
\newcommand{\R}{\mathbb{R}}

\newcommand{\D}{\Omega}
\newcommand{\Dc}{\overline{\Omega}}
\newcommand{\dbar}{\overline{\partial}}
\onehalfspace

\title[Compactness of Hankel Operators]{Compactness of Hankel Operators and Analytic Discs in the Boundary of Pseudoconvex Domains}

\author{\u{Z}eljko \u{C}u\u{c}kovi\'c}
\address[\u{Z}eljko \u{C}u\u{c}kovi\'c]{University of Toledo, Department of Mathematics, Toledo, OH 43606, USA}
\email{zcuckovi@math.utoledo.edu}

\author{S\"{o}nmez \c{S}ahuto\u{g}lu}
\address[S\"{o}nmez \c{S}ahuto\u{g}lu]{University of Michigan, Department of Mathematics, Ann Arbor, MI 48109, USA}
\email{sonmez@umich.edu}

\thanks{The second author is supported in part by NSF grant number DMS-0602191.}

\subjclass[2000]{Primary 47B35, 32W05}
\keywords{Hankel operators, $\dbar$-Neumann problem, pseudoconvex,  analytic discs}
\date{\today}
\begin{document}
\maketitle

\begin{abstract}
Using several complex variables techniques, we  investigate the interplay between the geometry of the boundary and  compactness of Hankel operators. Let $\beta$ be a function smooth  up to the boundary on a smooth bounded pseudoconvex domain $\D \subset \C^{n}.$ We show that, if $\D$ is convex  or the Levi form of the boundary of $\D$ is of rank at least $n-2,$ then compactness of the Hankel operator $H_{\beta}$  implies that $\beta$ is holomorphic ``along'' analytic discs in the boundary. Furthermore, when $\D$ is convex in $\C^2$ we show that the condition on $\beta$ is necessary and sufficient for compactness of $H_{\beta}.$
\end{abstract}

%%%%%%%%%%%%
%%%%%%%%%%%%
\section{Introduction}
%%%%%%%%%%%%
%%%%%%%%%%%%

Hankel operators form an important class of operators on spaces of holomorphic functions. Initially there were two descriptions of Hankel operators, one considered it as an operator on the one-sided sequence space $l^2$ into itself, and the other as an operator from the Hardy space $H^2$ of the unit disk into its orthogonal complement in $L^2$.  These operators are closely connected to problems in approximation theory as shown by now the famous work of Nehari \cite{Nehari57} on one hand, and Adamjan, Arov and Krein \cite{AdamjanEtall71} on the other.  These operators also have a close connection to Toeplitz operators, and the commutators of projections and multiplication operators on $L^2$. More about Hankel operators and related topics can be found \cite{Peller03}

Let $\Omega$ be a bounded domain in $\C^n$ and $dV$ denote the Lebesgue volume measure. The Bergman space $A^2(\Omega)$ is the closed subspace of $L^2(\Omega)$ consisting of holomorphic functions on $\Omega$.  The Bergman projection $P$ is the orthogonal projection from $L^2(\D)$ onto $A^2(\D)$ and  can be written explicitly as $Pf(z) = \int_{\Omega}K(z, w)f(w) dV(w),$ where $K(z, w)$ is the Bergman kernel of $\Omega$. For $\beta\in L^2(\Omega)$ we can define the Hankel operator $H_{\beta}$ from $A^2(\Omega)$ into $L^2(\Omega)$ by $H_{\beta}(g) = (Id - P)(\beta g) .$ In general, $H_{\beta}$ is only densely defined on $A^2(\Omega)$. When $\D$ is  a bounded pseudoconvex domain, Kohn's formula $P=Id-\dbar^*N\dbar$ ($N$ is the (bounded) inverse of complex Laplacian,  $\dbar\dbar^*+\dbar^*\dbar,$ and $\dbar^*$ is the Hilbert space adjoint of $\dbar$ on the square integrable $(0,1)$-forms on $\D$) implies that $H_{\beta}(f)=\dbar^*N\dbar(\beta f)=\dbar^{*}N(f\dbar\beta) $ for $f\in A^{2}(\D)$ and $\beta \in C^{1}(\Dc).$ This will be the  main tool in this paper as it will allow us to use several complex variables techniques to study Hankel operators. We refer the reader to \cite{ChenShawBook} for more information on the  $\dbar$-Neumann operator.

The study of the size estimates of Hankel operators on Bergman spaces has inspired a lot of work in the last 20 years. The first result  in the study of boundedness and compactness of Hankel operators was done by Axler \cite{Axler86} on the Bergman space of the open unit disk $\Delta$.  He showed that, for $\beta$ holomorphic on $\Delta$, $H_{\overline \beta}$ is bounded if and only if $\beta$ is in the Bloch space, and $H_{\overline \beta}$ is compact if and only if $\beta$ is in the little Bloch space.  In the case of a general symbol, Zhu \cite{Zhu87} showed the connection between size estimates of a Hankel operator and the mean oscillation of the symbol in the Bergman metric.  In \cite{BBCZ90}, Bekolle, Berger, Coburn and Zhu studied the same problem in the setting of bounded symmetric domains in $\C^n$ with the restriction that $H_{\beta}$ and $H_{\overline \beta}$ are simultaneously bounded and compact with $\beta\in L^2(\Omega)$.  Stroethoff and Zheng \cite{Stroethoff90IJM,Zheng89} independently gave a characterization for compactness of Hankel operators with bounded symbols on $\Delta.$ Later Stroethoff \cite{Stroethoff90JOT} generalized these results to the case of the open unit ball and polydisc in $\C^n.$   Luecking \cite{Luecking92} gave different criteria for boundedness and compactness of  $H_{\beta}$   on $A^p(\Omega)$ with $1 < p < \infty$.  Peloso \cite{Peloso94} extended Axler's result to Bergman spaces on smooth bounded strongly pseudoconvex domains.  For the same domains, Li \cite{Li94} characterized bounded and compact Hankel operators $H_{\beta}$ with symbols $\beta\in  L^2(\Omega)$.  Beatrous and Li \cite{BeatrousLi93} obtained related results for the commutators of multiplication operators and $P$ on more general domains, that include smooth bounded strongly pseudoconvex domains.

The novelty of our approach is that we put an emphasis on the interplay between the geometry of the domain and the symbols of Hankel operators.  Although, our symbols are more restricted the domains we consider are much more general and allow rich geometric structures. 

In several complex variables, compactness of the $\dbar$-Neumann operator has been  an active research area for the last couple of decades. We refer the reader to a very nice survey \cite{FuStraube01} for more information about compactness of the $\dbar$-Neumann operator. Compactness of the canonical solution operators for $\dbar$ on the unit disk has been discussed in \cite{Haslinger01}, where it was in fact shown that this operator restricted to $(0, 1)$-forms with holomorphic coefficients is a Hilbert-Schmidt operator. Fu and Straube \cite{FuStraube98} showed that presence of analytic discs in the boundary of a bounded convex domain in $\C^n$ is equivalent to the non-compactness of the $\overline\partial$-Neumann operator. The second author and Straube \cite{SahutogluStraube06} used their techniques to prove that analytic discs are obstructions for compactness of the $\dbar$-Neumann operator on smooth bounded  pseudoconvex domains in $\C^n$ whose Levi form has maximal rank. In $\C^2$ their result reduces to a folklore result of Catlin \cite{FuStraube98}.  

Given Kohn's formula it is natural to expect a strong relationship between Hankel operators and the $\dbar$-Neumann operator. The following fact confirms this expectation. Compactness of the $\dbar$-Neumann operator implies compactness of Hankel operators with symbols are smooth on the closure \cite{FuStraube01}.  Actually, the statement in \cite{FuStraube01} requires the symbol to have bounded first order derivatives. But any symbol that is continuous on the closure can be approximated uniformly by symbols that are smooth on the closure of the domain. Hence the resulting Hankel operators converge in norm preserving compactness. In this paper we show that the theory for compactness of Hankel operators is somewhat parallel to the theory of compactness of the $\dbar$-Neumann operator in terms of analytic structures in the boundary.  Previous work in this direction was done by Knirsch and Schneider  \cite{KnirschSchneider07}. 

Throughout the paper $b\D$ denotes the boundary of $\D.$ Our first result concerns smooth bounded pseudoconvex domains in $\C^{n}.$ 

\begin{theorem}\label{ThmCn}
Let $\D$ be a smooth bounded pseudoconvex domain in $\C^n$ for $n\geq 2$ and $\beta\in C^{\infty}(\Dc).$ Assume that the Levi form of $b\D$ is of rank at least $n-2.$ If $H_{\beta}$ is compact on $A^2(\D),$ then  $\beta \circ f$ is holomorphic for any holomorphic function $f:\Delta \to b\D .$
 \end{theorem}

\begin{remark}\label{RemarkAlong}
We note that the statement ``$\beta \circ f$ is holomorphic'' can be interpreted as meaning  that $\beta$ is holomorphic ``along'' $M=f( \Delta).$ However it may not be holomorphic in the transversal directions.
\end{remark}

\begin{remark}
One can check that the proof of Theorem \ref{ThmCn} shows  that compactness of  $H_{\beta}$ on $A^2(\D)$ for $\beta \in C^{\infty}(\Dc)$ still implies that  $\beta \circ f$ is holomorphic for any holomorphic function $f:\Delta \to b\D $ when $\D$ satisfies the following property: If the Levi form of $b\D$ is of rank $k$ for $0\leq k\leq n-1$  at $p,$ then there exists a $n-k-1$ dimensional complex manifold in $b\D$ through $p.$ 
\end{remark}

Since in $\C^2$ the Levi form has only one eigenvalue the condition on the Levi form in Theorem \ref{ThmCn} is always satisfied. Therefore, for $n=2$ we have the following corollary. 

\begin{corollary}\label{ThmC2}
Let $\D$ be a smooth bounded pseudoconvex domain in $\C^{2}$ and $\beta\in C^{\infty}(\Dc).$  If $H_{\beta}$ is compact on $A^{2}(\D)$ then   $\beta \circ f$ is holomorphic for any holomorphic function $f:\Delta \to b\D .$
\end{corollary}

For convex domains in $\C^{n}$ we prove the same result without any restriction on the Levi form. 

\begin{theorem}\label{ThmConvexCn}
Let $\D$ be a smooth bounded convex domain in $\C^{n}$ for $n\geq 2$ and $\beta\in C^{\infty}(\Dc).$ Assume that  $H_{\beta}$ is compact on $A^{2}(\D).$  Then   $\beta \circ f$ is holomorphic for any holomorphic function $f:\Delta \to b\D .$
\end{theorem}

In the following theorem we show that, when $\D$ is convex in $\C^2,$ the converse of Theorem \ref{ThmCn} is true. 

\begin{theorem} \label{ThmConvex}
Let $\D$ be a smooth bounded convex domain in $\C^{2}$ and $\beta\in C^{\infty}(\Dc).$   If $\beta \circ f$ is holomorphic for any holomorphic $f:\Delta\to b\D,$ then  $H_{\beta}$ is compact.
\end{theorem}

Combining Corollary \ref{ThmC2} (or Theorem \ref{ThmConvexCn}) and Theorem \ref{ThmConvex} we get a necessary and sufficient condition for compactness of $H_{\beta}$ for convex domains in $\C^2.$

\begin{corollary}\label{CorConvex}
Let $\D$ be a smooth bounded  convex domain in $\C^{2}$ and $\beta\in C^{\infty}(\Dc)$. Then  $H_{\beta}$ is compact if and only if $\beta \circ f$ is holomorphic for any holomorphic $f:\Delta\to b\D.$ 
\end{corollary}

\begin{remark}
We note that \cite{MatheosThesis} constructed a smooth bounded pseudoconvex complete Hartogs domain $\D$ in $\C^{2}$ that has no analytic disk in its boundary, yet it does not have a compact $\dbar$-Neumann operator. It would be interesting to know whether there exists a symbol $\beta\in C^{\infty}(\Dc)$ such that the Hankel operator $H^{\D}_{\beta}$ is not compact on $A^2(\D).$
\end{remark}

\begin{remark}
We would like to take this opportunity to point out an inaccuracy. 
Knirsch and Schneider  \cite[Proposition 1]{KnirschSchneider07} claim that if there is an affine disk in the boundary of a bounded convex domain in $\C^{n},$ then the Hankel operator $H_{\bar z_i^m}$ is not compact for  $i=1,2,\ldots,n$ and any positive integer $m$  where $z_i$ is the ith coordinate function. They correctly prove the result when the disk lies in $z_{1}$-coordinate and claim that the proof for $i=2,3,\ldots, n$ is similar. However, Theorem \ref{ThmConvex}  implies that if $\D$ is a smooth bounded convex domain in $\C^2$ and  the set of weakly pseudoconvex points form a disc in $z_{1}$-coordinate, then  $H_{\bar z_2}$ is compact. 
\end{remark}

\begin{remark}
For simplicity we assume that the domains have $C^{\infty}$-smooth boundary and the symbols are smooth up to the boundary. However, one can check that the proofs work under weaker but reasonable smoothness assumptions.
\end{remark}

\begin{remark}
Recently, \c{C}elik and Straube \cite{CelikStraube} studied compactness multipliers for the $\dbar$-Neumann problem (we refer the reader to \cite{CelikStraube} for the definition and some properties of compactness multipliers of the $\dbar$-Neumann problem). This notion is related to that of a symbol of a compact Hankel operator, but there are differences. 
First of all, the $\dbar$-Neumann operator $N$ is applied to square integrable forms and  compactness multipliers are applied after $N$.  In case of Hankel operators,  however, one can think of the $(0,1)$-form $\overline{\partial}\beta$ as acting as a ``pre-multiplier'' (acting before the canonical solution operator $\dbar^*N$) on the Bergman space which is  more rigid than the space of $L^2$ forms. Secondly,  \c{C}elik and Straube proved that on a bounded convex domain, a function that is continuous on the closure is a compactness multiplier if and only if the function vanishes on all the (nontrivial) analytic discs in the boundary. One can show that such symbols produce compact Hankel operators. However, for smooth bounded convex domains in $\C^2$,  a symbol smooth on the closure  produces a compact Hankel operator if and only if the symbol is holomorphic along (see Remark \ref{RemarkAlong}) analytic discs in boundary. (That is, the complex tangential component of the pre-multiplier on any analytic disc in the boundary vanishes). In general, these connections are not well understood. For example, the following question is still open:
\end{remark}

\begin{question}
Assume that $\D$ is a smooth bounded pseudoconvex domain in $\C^n$ and $\beta\in C(\Dc)$ is a compactness multiplier for the $\dbar$-Neumann operator on $L^2_{(0,1)}(\D).$ Is $H_{\beta}$ compact on the Bergman space on $\D?$
\end{question}

%%%%%%%%%%%
%%%%%%%%%%%
\section{Proof of Theorem \ref{ThmCn} and Theorem \ref{ThmConvexCn}}\label{ProofThmCn}
%%%%%%%%%%%%
%%%%%%%%%%%%

Let $\Delta=\Delta_{1}$ denote the unit open disc in $\C, \Delta_{r}$ denote the disc in $\C$ centered at the origin with radius $r,$ and $\Delta^{k}_{r}$ denote the polydisc in $\C^{k}$ of multiradius $(r,\cdots,r).$ We will be using Hankel operators defined on different domains. So to be more precise, let $H^{\D}_{\phi}$ denote the Hankel operator on $\D$ with symbol $\phi$ and $R_{U}$ be the restriction operator onto $U.$ Furthermore, the Bergman projection on $U$ will be denoted by $P_{U}.$  First we will start with a proposition that will allow us to ``localize'' the proofs. 

In the proofs below we will use generalized constants. That is $A\lesssim B$ will mean that there exists a constant $c>0$ that is independent of the quantities of interest such that $A\leq cB .$ At each step the constant $c$ may change but it will stay independent of the quantities of interest.

\begin{proposition}\label{prop1}
Let $\D$ be a bounded pseudoconvex domain in $\C^{n}$ and $\phi \in L^{\infty}(\D) .$ Then  
\begin{itemize}
\item[i)] If $H^{\D}_{\phi}$  is compact  on $A^{2}(\D)$ then for every $p\in b\D$ and $U$ an open neighborhood of $p$ such that $U\cap \D$ is a domain, $H^{U\cap \D}_{R_{U\cap \D}(\phi)}R_{U\cap \D}$ is compact on $A^{2}(\D).$
\item[ii)] If for every $p\in b\D$ there exists an open neighborhood $U$ of $p$ such that $U\cap \D$ is a domain, and $H^{U\cap \D}_{R_{U\cap \D}(\phi)}R_{U\cap \D}$ is compact on $A^{2}(\D),$ then $H^{\D}_{\phi}$  is compact  on $A^{2}(\D).$
\end{itemize}
\end{proposition}

\begin{proof}
Let us prove i) first. For $f\in A^{2}(\D)$  we have 
\begin{eqnarray*}
 (Id_{U\cap\D}-P_{U\cap\D})R_{U\cap \D} H^{\D}_{\phi}(f)&=& (Id_{U\cap\D}-P_{U\cap\D})R_{U\cap \D}(\phi f-P_{\D}(\phi f))\\
&=&H^{U\cap \D}_{R_{U\cap \D}(\phi)}R_{U\cap \D}(f)+P_{U\cap\D}R_{U\cap \D}P_{\D}(\phi f)-R_{U\cap \D}P_{\D}(\phi f)\\
&=&H^{U\cap \D}_{R_{U\cap \D}(\phi)}R_{U\cap \D}(f).
\end{eqnarray*}
In the last equality we used the fact that $P_{U\cap\D}R_{U\cap \D}P_{\D}=R_{U\cap \D}P_{\D}$ on $L^{2}(\D) .$  Hence 
\[ (Id_{U\cap\D}-P_{U\cap\D})R_{U\cap \D} H^{\D}_{\phi}(f)= H^{U\cap \D}_{R_{U\cap \D}(\phi)}R_{U\cap \D}(f).\]
Therefore, if $H^{\D}_{\phi}$ is compact, then $H^{U\cap \D}_{R_{U\cap \D}(\phi)}R_{U\cap \D}$ is also compact. 

To prove ii) let us choose $\{p_{1},\ldots,p_{m}\}\subset b\D$ and open sets  $U_{1},\ldots,U_{m}$ such that
\begin{itemize}
\item[i)] $U_j$ is a neighborhood of $p_j$ and $U_{j}\cap \D$ is a domain for  $j=1,\ldots,m,$
\item[ii)] $b\D \subset \cup_{j=1}^{m} U_{j},$   
\item[iii)] $S_{j}=H^{U_{j}\cap \D}_{R_{U_{j}\cap \D}(\phi)}R_{U_{j}\cap \D}$ is compact for $j=1,\ldots,m.$ 
\end{itemize}

Let $U_0=\D, S_0=H^{\D}_{\phi},$ and $\{\chi_{j}:j=0,\ldots,m\}$ be a $C^{\infty}$-smooth partition of unity subject to $\{U_{j}:j=0,\ldots,m\} .$ Then for $f\in A^2(\D)$
\begin{eqnarray*}
\dbar \left(\sum_{j=0}^{m} \chi_{j}S_{j}(f)\right)&=& \sum_{j=0}^{m} (\dbar\chi_{j})S_{j}(f)+\sum_{j=0}^{m} \chi_{j}\dbar S_{j}(f)\\
&=&\sum_{j=0}^{m} (\dbar\chi_{j})S_{j}(f)+\sum_{j=0}^{m} \chi_{j}(\dbar \phi) f\\
&=&\sum_{j=0}^{m} (\dbar\chi_{j})S_{j}(f)+(\dbar \phi) f.
\end{eqnarray*}
Hence, since $\dbar \left(\sum_{j=0}^{m} \chi_{j}S_{j}(f)\right)$ and $(\dbar \phi) f$ are $\dbar$-closed we conclude that $\sum_{j=0}^{m} (\dbar\chi_{j})S_{j}(f)$ is $\dbar$-closed. Let 
\begin{equation} \label{eqnprop}
 S=\sum_{j=0}^{m} \chi_{j}S_{j}-\dbar^{*}N^{\D} \sum_{j=0}^{m}(\dbar\chi_{j})S_{j}.
\end{equation}

We write $\chi_0S_0(f)$ as $\chi_0\phi f-\chi_0P_{\D}(\phi f) $ and choose a bounded sequence $\{f_j\}$ in $A^2(\D).$ Let $K$ be a compact set in $\D$ that contains a neighborhood of the support of $\chi_0.$ Cauchy integral formula and Montel's theorem imply that $\{f_j\}$ and $\{P_{\D}(\phi f_j)\}$ have uniformly convergent subsequences on $K.$ Then $\{\chi_0\phi f_ j\}$ and $\{\chi_0P_{\D}(\phi f_j)\} $ have convergent subsequences in $L^2(\D).$ That is, the operator $\chi_0S_0$ is compact on $ A^2(\D).$ Similarly, $(\dbar \chi_0)S_0$ is compact as well. We remind the reader that we assumed that $S_{j}$ is compact for $j=1,\ldots,m$ and $\dbar^{*}N^{\D}$ is continuous on bounded pseudoconvex domains. Therefore, \eqref{eqnprop} implies that $S$ is a compact operator and $\dbar S(f) =(\dbar\phi) f.$ To get the Hankel operator we project onto the complement of $A^{2}(\D).$ Hence using $H_{\phi}^{\D} =(Id_{\D}-P_{\D})S$ we conclude that  $H_{\phi}^{\D}$ is compact on $A^2(\D).$
\end{proof}

 \begin{lemma}\label{lem2}
 Let $\D_{1}$ and $\D_{2}$ be two bounded pseudoconvex domains in $\C^{n},$ 
$\phi\in C^{\infty}(\Dc_{2}),$ and  $F:\D_{1}\to \D_{2}$ be a biholomorphism
that has a smooth extension up to the boundary. Assume that $H_{\phi}^{\D_{2}}$ is compact on $A^2(\D_2)$. Then $H_{\phi\circ F}^{\D_{1}}$ is compact on 
$A^2(\D_1) .$ 
\end{lemma}

\begin{proof}
Let  $g\in A^2(\D_1), f=g\circ F^{-1}, u=\dbar^*N^{\D_2}\dbar \phi f ,$ and $w=u\circ F=F^{*}(u).$  Then $ f\in A^2(\D_2), u=H^{\D_2}_{\phi}(f),$ and
\[\dbar w=  \dbar F^{*}(u)=F^{*}(\dbar u)=F^{*}(f \dbar \phi)=(f\circ F)\dbar(\phi\circ F).\]
So $\dbar (u\circ F)=(f\circ F)\dbar (\phi\circ F)$ on $\D_1$ and $\dbar^{*} N^{\D_1}\dbar (u\circ F)$ is the canonical solution for $\dbar w=(f\circ F)\dbar (\phi\circ F)$ on $\D_1.$ Then   
\[H^{\D_1}_{\phi\circ F}(g)=H^{\D_1}_{\phi\circ F}(f\circ F)=\dbar^{*}
N^{\D_1}\dbar (u\circ F)=\dbar^{*} N^{\D_1}\dbar (F^* H_{\phi}^{\D_2}((F^{-1})^{*}(g))).\]
Therefore, $H^{\D_1}_{\phi\circ F}$ is a composition of $H^{\D_2}_{\phi}$ with continuous operators $\dbar^{*}N^{\D_1}\dbar, F^* ,$ and $(F^{-1})^{*}.$ Then since $H_{\phi}^{\D_2}$ is assumed to be compact on $A^2(\D_2)$ we conclude that $H_{\phi\circ F}^{\D_{1}}$ is compact on $A^2(\D_1).$  
\end{proof}
 
Let $d_{b\D}(z)$ be the function defined on $\D$ that measures the (minimal) distance from $z$ to $b\D.$ The Bergman kernel function of $\D$ satisfies the following relation on the diagonal of $\D\times \D$
\[ K_{\D} (z, z) = \sup\{ |f (z)|^2: f \in A^2 (\D), \| f \|_{L^{2}(\D)} \leq 1\}.\]
The following proposition appeared in \cite{Fu94} for general pseudoconvex domains in $\C^n$ and in \cite{SahutogluThesis} in the following form. 
\begin{proposition}\label{PropFu}
 Let $\D$ be a bounded pseudoconvex domain in $\C^n$ with $C^2$-boundary
near $p\in b\D.$ If the Levi form is of rank $k$ at $p,$ then there exist a constant $C > 0$ and a neighborhood $U$ of $p$ such that
\[ K_{\D} (z, z) \geq  \frac{C}{(d_{b\D}(z))^{k+2} }\text{ for } z \in U \cap \D. \]
\end{proposition}

\begin{proof}[Proof of Theorem \ref{ThmC2}] 
We will prove a stronger result. The proof will go along the lines of the proof of Theorem 1 in \cite{SahutogluStraube06} and the proof of $(1)\Rightarrow (2)$ in \cite{FuStraube98} with some additional work. The same strategy has appeared in \cite{Catlin81,DiederichPflug81,SahutogluThesis}. Let us assume that 
\begin{itemize}
 \item[i.] $\D$ is a smooth bounded pseudoconvex domain in $\C^n$ and $p\in b\D,$
\item[ii.] the Levi form of $b\D$ is of rank $k$ at $p$ through which there exists a $n-k-1$ dimensional complex manifold in $b\D,$
\item[iii.] there exists non-constant holomorphic mapping $f:\Delta^{n-k-1} \to b\D$ and $q\in \Delta$ such that $f(q)=p,$ $Df(q)$ is full rank ($Df$ is the Jacobian of $f$), and $\dbar(\beta\circ f)(q)\neq 0,$
\item[iv.] $H_{\beta}$ is compact.
\end{itemize}
Lemma 1 in \cite{SahutogluStraube06} implies that there exist a neighborhood $V$ of $p$ and a local holomorphic change of coordinates $G$ on $V$ so that $G(p)=0,$ positive $y_n$-axis is the outward normal direction to the boundary of  $\D_1=G(V\cap \D)$ at every point of $M =\{z\in \Delta^n: z_{n-k}=\cdots=z_n=0 \}\subset b\D_{1} .$ 

Let $z=(z',z'')$ where $z'=(z_1,\ldots, z_{n-k-1})$ and $z''=(z_{n-k},\ldots,z_n).$ We define $L$ to be the $k+1$ (complex) dimensional slice of $\D_1$ that passes through the origin and is orthogonal to $M.$ That is, $L=\{z''\in \C^{k+1}:(0,z'')\in \D_{1}\}.$ So $L$ is strongly pseudoconvex at the origin when $k\geq 1$ and is  a domain in $\C$ when $k=0.$   Then there exists $0<\lambda<1$ such that $ M_{1}\times  L_{1} \subset \D_{1}, $ where $ L_{1}$ is a ball in $\C^{k+1}$ centered at $(0,\ldots,0,-\lambda)$ with radius $\lambda$  and $M_{1}=\frac{1}{2} M.$ For every $j$ we choose $p_j=(0,\ldots,0,-1/j)\in M_{1}\times  L_{1}.$ We take the liberty to abuse the notation and consider $p_j=(0,\ldots,0,-1/j)\in  L_{1}.$ Now we define $q_{j}=G^{-1}(p_{j})\in V\cap \D$ and 
\[f_j(z)=\frac{K_{\D}(z,q_j)}{\sqrt{K_{\D}(q_j,q_j)}}.\]
One can check that $\{f_j\}$ is a bounded sequence of square integrable functions on $\D$ that converges to zero locally uniformly. Let us define
$\alpha_{j}=f_j\circ G^{-1}$ and $\beta_{1}=\beta\circ G^{-1}.$ Then i) in Proposition \ref{prop1} implies that $H^{V\cap \D}_{R_{V\cap \D}(\beta)}R_{V\cap \D}$ is compact. In turn, Lemma \ref{lem2} implies that $H^{\D_{1}}_{\beta_{1}}$ is compact. Hence $\{H^{\D_{1}}_{\beta_{1}}(\alpha_{j})\}$ has a convergent subsequence. The strategy for the rest of the proof  will be to prove that $\{H^{\D_{1}}_{\beta_{1}}(\alpha_{j})\}$ has no convergent subsequence. Hence, getting a contradiction. 

Since  $\dbar (\beta \circ f )(q) \neq 0$ without loss of generality we may assume that $\left|\frac{\partial \beta_{1}}{\partial \bar z_{1}}\right|>0$ at the origin. Then  there exist $0<r<1$ and   a smooth function $0\leq \chi\leq 1$  on real numbers such that 
\begin{itemize}
 \item[i.] $\Delta^{n-k-1}_r\subset M_1,$ 
\item[ii.] $\chi(t)= 1$ for $|t|\leq r/2, \chi (t)= 0$ for  $|t|\geq 3r/5,$ 
\item[iii.] $\left|\frac{\partial \beta_{1}}{\partial \bar z_{1}}\right|>0$ on $\Delta_{r}^{n}.$ 
\end{itemize}
Then  $C=\int_{|z_1|<3r/4}\chi(|z_1|)dV(z_1) > 0.$ Let us define $\gamma$ on $\D_{1}$ so that
\[\gamma(z)\frac{\partial\beta_{1}(z)}{\partial
\bar z_{1}}=\chi(|z_1|)\cdots\chi(|z_n|).\] 
and  $\langle .,. \rangle$ denote the standard pointwise inner product on forms in $\C.$  Furthermore, let $z=(z_{1},w)$ where $w=(z_{2},\ldots,z_{n})$ and $\alpha\in A^{2}(\D_{1}).$ Then using the mean value property for a holomorphic function $\alpha$ and for fixed $w\in \Delta^{n-1}_{3r/4}$ so that $\Delta_{r}\times \{w\} \subset M_{1}\times L_{1}$ we get 
\begin{eqnarray*}
 C\alpha(0,w)&=&\int_{|z_1|<3r/4}
\chi(|z_1|)\alpha(z_1,w)dV(z_1)\\
&=&\int_{|z_1|<3r/4}\gamma(z_1,w)\frac{\partial\beta_{1}
(z_1,w)}{\partial\bar
z_{1}}\alpha(z_1,w)dV(z_1)
\end{eqnarray*}
On the other hand, 
\begin{eqnarray*}
\int_{|z_1|<3r/4}\gamma(z_1,w)\frac{\partial\beta_{1}(z_1,w)}{\partial\bar z_{1}}\alpha(z_1,w)dV(z_1)
&=&\int_{|z_1|<3r/4}\langle\alpha\dbar \beta_{1} , \bar \gamma d\bar z_1\rangle dV(z_1)\\
&=&\int_{|z_1|<3r/4} \langle \dbar\dbar^{*}N^{\D_{1}}(\alpha\dbar \beta_{1}),\bar \gamma d\bar z_1\rangle dV(z_1)\\
&=&\int_{|z_1|<3r/4}\frac{\partial H^{\D_{1}}_{\beta_{1}}(\alpha)}{\partial \bar z_{1}} \gamma  dV(z_1)\\
&=&-\int_{|z_1|<3r/4}H^{\D_{1}}_{\beta_{1}}(\alpha) \frac{\partial \gamma}{\partial \bar z_{1}}  dV(z_1).
\end{eqnarray*}
Therefore, we have
\begin{eqnarray*} 
|\alpha(0,w)|
&\lesssim&
\left(\int_{|z_1|<3r/4}|H^{\D_{1}}_{\beta_{1}}(\alpha
)|^{2}dV(z_1)\right)^{1/2}.
\end{eqnarray*}
If we square both sides  we get  
\begin{eqnarray*} 
\left| \alpha(0,w) \right|^2\lesssim
\int_{|z_1|<3r/4}|H^{\D_{1}}_{\beta_{1}}(\alpha
)(z_{1},w)|^{2}dV(z_1).
\end{eqnarray*}
Since $\left| \alpha(0,w) \right|^2$ is subharmonic when we integrate over $(z_{2},\cdots, z_{n-k-1})\in \Delta^{n-k-2}_{3r/4},$ we get
\begin{eqnarray}\label{EqnSlice} 
\left| \alpha(0,z'') \right|^2\lesssim
\int_{z'\in \Delta^{n-k-1}_{3r/4}}|H^{\D_{1}}_{\beta_{1}}(\alpha
)(z',z'')|^{2}dV(z').
\end{eqnarray}
The above inequality applied to $\alpha_{j}$ implies that $\alpha_{j}|_{L_{1}}\in L^{2}(L_{1})$. Now we use the reproducing property of $K_{L_{1}}$ on $\alpha_{j}|_{L_{1}}$ to get 
\[\alpha_{j}(p_{j})=\int_{L_{1}}K_{L_{1}}(p_{j},z)\alpha_{j}|_{L_1}(z)dV(z) .\]
Cauchy-Schwartz inequality implies that  $| \alpha_{j}(p_{j})|\leq \|\alpha_{j}|_{L_{1}}\|_{L^{2}(L_{1})}\| K_{L_{1}}(p_{j},.)\|_{L^{2}(L_{1})} $. On the other hand $\| K_{L_{1}}(p_{j},.)\|_{L^{2}(L_{1})}=\sqrt{ K_{L_{1}}(p_{j},p_{j})}. $ So we have 
\[ \|\alpha_{j}|_{L_{1}}\|_{L^{2}(L_{1})} \geq \frac{| \alpha_{j}(p_{j})|}{ \sqrt{K_{L_{1}}(p_{j},p_{j})}}=\sqrt{\frac{K_{\D}(q_{j},q_{j})}{K_{L_{1}}(p_{j},p_{j})}}.\]
Since $L_{1}$ is a ball in $\C^{k+1}$ and the rank of the Levi form for $\D$ (and hence for $\D_{1}$) is at least $k,$ the asymptotics of the Bergman kernel on balls and  Proposition \ref{PropFu} imply  the following inequalities: 
\[ \frac{1}{(d_{bL_{1}}(p_{j}))^{k+2}}\lesssim K_{L_{1}}(p_{j},p_{j}) \lesssim \frac{1}{(d_{bL_{1}}(p_{j}))^{k+2}},\] 
\[\frac{1}{(d_{b\D}(q_{j}))^{k+2}}\lesssim K_{\D}(q_{j},q_{j}).\]
We note that $p_{j}$ and $q_{j}$ are related by a diffeomorphism. So for large enough $j$ $d_{b\D_{1}}(p_{j})=d_{bL_{1}}(p_{j})$ and  they are comparable to $d_{b\D}(q_{j}).$ Therefore, there exists $\tilde \xi>0$ such that $\tilde \xi <\|\alpha_{j}|_{L_{1}}\|_{L^{2}(L_{1})}$ for all $j$. Since $\{\alpha_{j}\}$ converges to 0 locally uniformly this implies that $\{\alpha_{j}|_{L_{1}}\}$ has no convergent subsequence in $L^{2}(L_{1}).$ Also  \eqref{EqnSlice} applied to $\alpha_j-\alpha_k$ implies that  
\[\|\alpha_{j}|_{L_{1}}-\alpha_{k}|_{L_{1}}\|_{L^{2}(L_{1})} \lesssim \|H^{\D_{1}}_{\beta_{1}}(\alpha_ {j}-\alpha_{k})\|_{L^{2}(\D_{1})}.\]
Hence $\{H^{\D_{1}}_{\beta_{1}}(\alpha_ {j})\} $ has no convergent subsequence in $L^{2}(\D_{1}).$  Therefore, we  have reached a contradiction completing the first proof of Theorem \ref{ThmC2}.
\end{proof}

A weaker version of the following lemma appeared in \cite{FuStraube98}.

\begin{lemma}\label{LemAffine}
Let $\D$ be a convex domain in $\C^{n}$ and $f:\Delta\to b\D$ be a non-constant holomorphic map. Then the convex hull of $f(\Delta)$ is an  affine analytic variety contained in $b\D.$ 
\end{lemma}
\begin{proof}
Let $K$ be the convex hull of $f(\Delta)$ in $\C^n.$  First we will show that $K$ is an analytic affine variety. By definition $K$ is an affine set in $\C^n.$  Let $F(z,w,t)=tf(z)+(1-t)f(w)$ for $(z,w)\in \Delta^2$ and $0<t<1.$ One can check that  
\[K=\{F(z,w,t):(z,w,t)\in \Delta^2\times (0,1)\}.\] 
If $K$ is open in $\C^n$ then we are done. Otherwise, there exists $p\in K$ which is a boundary point and, by convexity, there exists $(z_0,w_0,t_0)\in\Delta^2\times (0,1)$ such that after possible rotation and translation $p=F(z_0,w_0,t_0)$ is the origin and $K\subset \{x_n\leq 0\}.$ Let us define $g=Re(z_{n}\circ F):\Delta^2\times(0,1)\to \R.$  Then  $g(z_0,w_0,t_0)=0$ and $g(\Delta^2\times(0,1))\subset \{x\in \R: x\leq 0\}.$ Maximum principle applied to the harmonic function $g$ implies that $g\equiv 0.$  Hence  $K \subset \{x_n=0\}.$  Since $f$ is holomorphic, $f'$ must stay in the complex tangent subspace of $\{x_n=0\}.$ That is,
\begin{equation} \label{EqnConvex}
f'(p)\subset \text{span}\left\{\frac{\partial}{\partial z_1},\ldots, \frac{\partial}{\partial z_{n-1}}\right\} \text{ for every } p\in \Delta. 
\end{equation}
Now it is easy to see that \eqref{EqnConvex} implies that $K\subset \{z_n=0\}.$ So we have demonstrated that if $K$ is not an $n$ dimensional analytic affine variety then it is contained in an $n-1$ dimensional analytic affine variety. We use the above argument multiple times if necessary to show that $K$ is open in an analytic affine variety. Hence $K$ is an analytic affine variety.

Now we will show that $K$ is contained in $b\D.$ Since $K$  and $\D$ are convex after some possible rotation and translation, we can assume that $f(0)$ is the origin and $f(\Delta)\subset \Dc \subset \{x_{n}\leq 0\} .$ Since $\emptyset \neq f(\Delta)\subset K\cap b\D$ the set $K$ is not an open set in $\C^{n}$. Then, as in the above paragraph, one can show that $K\subset \{x_n=0\}\cap \Dc \subset b\D.$ This completes the proof of the lemma.
\end{proof}

\begin{proof}[Proof of Theorem \ref{ThmConvexCn}] 
The proof will be very similar to the first part and the proof of $(1) \Rightarrow (2)$ in \cite{FuStraube98}. So we will just sketch the proof and point out differences. Let us assume that $H^{\D}_{\beta}$ is compact and that there exists a nonconstant holomorphic map $ f:\Delta \to b\D.$ We can choose $p\in \Delta$ such that $|\dbar (\beta\circ  f) (p)|>0.$  By applying translation and rotation, if necessary,  we may assume that $f(p)=0, f'(p)=(1,0,\ldots,0),$ and positive $x_{n}$-axis is the outward normal for $b\D$ at $0.$ Using Lemma \ref{LemAffine} with scaling, if necessary, we may assume that  $\{(z,0,\ldots,0)\in \C^n:|z|\leq 1\} \subset b\D$ and  $|\frac{\partial \beta(0)}{\partial \bar z_{1}}|>0 .$ We define $$L=\{(z_2,\ldots,z_{n})\in \C^{n-1}:(0,z_{2},\ldots,z_{n})\in \D\}, $$   $p_j=(0,\ldots,-1/j)\in L,$ and  $f_j(z)=\frac{K_{L}(z,q_j)}{\sqrt{K_{L}(q_j,q_j)}}.$ Using the proof of $(1) \Rightarrow (2)$ in \cite{FuStraube98} one can easily prove that $\{f_j\}$ is a bounded sequence in $A^2(L)$ such that  $\{R_{\lambda L}(f_j)\},$ the restricted sequence of $\{f_j\}$ to $\lambda L,$ has  no convergent subsequence in $A^2(\lambda L)$ for any $0<\lambda<1$.  Then for each $j$ we extend $f_j$ to $\D$ using Ohsawa-Takegoshi theorem \cite{OhsawaTakegoshi87} to get a bounded sequence $\{\alpha_j\}$ on $A^2(\D).$ Using similar arguments as in the proof of Theorem \ref{ThmC2}  and the fact that $\Delta_{1/2}\times \frac{1}{2}L \subset \D$ (this follows from convexity of $\D$) one can show that 
\[\|f_j-f_k\|_{L^{2}(\frac{1}{2}L)} \lesssim \|H^{\D}_{\beta}(\alpha_j-\alpha_k)\|_{L^{2}(\D)}.\]  
This contradicts the assumption that  $H^{\D}_{\beta}$ is compact.
\end{proof}

%%%%%%%%%%%%%%%%
%%%%%%%%%%%%%%%%
\section{Proof of Theorem \ref{ThmConvex}}\label{ProofThmConvex}
%%%%%%%%%%%%%%%
%%%%%%%%%%%%%%%

We refer the reader to \cite[Proposition V.2.3]{D`AngeloIneqBook} for a proof of the following standard lemma. 
\begin{lemma}\label{CompactEst}
Let $T:X\to Y$ be a linear operator between two Hilbert spaces $X$ and $Y$. Then $T$ is compact if and only if  for every $\epsilon>0$ there exist a compact operator $K_{\epsilon}:X\to Y$ and  $C_{\epsilon}>0$ so that  
\[\|T(h)\|_Y\leq \epsilon\|h\|_X+C_{\epsilon}\|K_{\epsilon}(h)\|_Y \textrm{ for } h\in X.\]
\end{lemma}

\begin{proof}[Proof of Theorem \ref{ThmConvex}]  Let $K$ denote the closure of the union of all analytic discs in $b\D.$ Let us choose a defining function $\rho$ for $\D$ so that $\|\nabla \rho\|=1$. Let $\beta=\beta_1+i\beta_2,$  
\begin{eqnarray*}
\nu=\sum_{j=1}^{2}\frac{\partial \rho}{\partial  x_j}\frac{\partial }{\partial x_j}+ \frac{\partial \rho}{\partial y_j} \frac{\partial}{\partial y_j}, \text{ and }
T=\sum_{j=1}^{2}\frac{\partial \rho}{\partial  x_j}\frac{\partial }{\partial y_j}-\frac{\partial \rho}{\partial y_j} \frac{\partial}{\partial x_j}.
\end{eqnarray*}
For sufficiently small $\epsilon$ and $\xi\in b\D,$ let us define  
\begin{eqnarray*}
\widetilde{\beta_1}(\xi+\epsilon\nu(\xi))=\beta_1(\xi)+\epsilon T(\beta_2)(\xi)\text{ and }
\widetilde{\beta_2}(\xi+\epsilon\nu(\xi))=\beta_2(\xi)-\epsilon T(\beta_1)(\xi).
\end{eqnarray*}
Then $\widetilde \beta=\widetilde{\beta_1}+i\widetilde{\beta_2}$ is a smooth function in a neighborhood of $b\D$ and it is equal to $\beta$ on the boundary of $\D.$ Let us extend $\widetilde\beta$ as a smooth function on $\Dc$ and still call it $\widetilde{\beta}.$ One can check that $(\nu+iT)(\widetilde{\beta})=0$ on $b\D.$ That is, in some sense $\widetilde \beta$ is holomorphic along complex normal direction on the boundary.  Let us define $\widehat{\beta}=\beta-\widetilde{\beta}$ on $\Dc.$ Then $\widetilde{\beta}$ and $\widehat{\beta}$ are smooth functions on $\Dc$ such that $\widehat{\beta}=0$ on $b\D$ and $\widetilde{\beta}$ is holomorphic on $K.$ Montel's theorem together with the fact that $\widehat{\beta}$ can be approximated by smooth functions supported away from the boundary imply that $H^{\D}_{\widehat{\beta}}$ is compact on $A^2(\D).$ In the rest of the proof we will show that $H^{\D}_{\widetilde{\beta}}$ is compact on $A^2(\D).$ Let $\{\psi_j\}$ be a sequence in  $C^{\infty}_{(0,1)}(\Dc)$  such that $\psi_j=0$ in a neighborhood of $K$ for all $j$ and  $\psi_j$ converges to $\dbar \widetilde{\beta}$ uniformly on $\Dc.$ On the boundary, $\psi_j$'s are supported on sets that satisfy property $(P)$ (see \cite{FuStraube98} when $\D$ is convex).

In the following calculation $\langle.,.\rangle_{L^{2}(\D)}$ denotes the $L^2$ inner product on $\D$ and $N=N^{\D}.$ Now we will show that $H^{\D}_{\widetilde{\beta}}$ is compact. Let  $g\in A^{2}(\D).$ Then we have 
\begin{eqnarray*}
\langle \dbar^{*}N(g \dbar \widetilde{\beta}),\dbar^{*}N(g\dbar\widetilde{\beta}) \rangle_{L^{2}(\D)}&=& \langle N(g \dbar  \widetilde{\beta}), g\dbar \widetilde{\beta} \rangle_{L^{2}(\D)}\\
&=& \langle N(g \dbar  \widetilde{\beta}), g(\dbar \widetilde{\beta} -\psi_j)\rangle_{L^{2}(\D)} +\langle N(g \dbar  \widetilde{\beta}), g\psi_j \rangle_{L^{2}(\D)}.
\end{eqnarray*}
Let us fix $\psi_j$. We choose $\psi\in C^{\infty}(\Dc)$ such that $0\leq \psi\leq 1,\psi\equiv 1$ on the support of $\psi_j$ and $\psi$ is supported away from $K.$ Then  for $g\in A^{2}(\D)$ we have 
\begin{equation} \label{EqnComp1}
|\langle N(g \dbar  \widetilde{\beta}), g\psi_j \rangle_{L^{2}(\D)}| = |\langle \psi N(g \dbar  \widetilde{\beta}), g\psi_j \rangle_{L^{2}(\D)}| \leq  \|\psi N(g \dbar  \widetilde{\beta})\|_{L^{2}(\D)} \|g\|_{L^{2}(\D)}.
\end{equation}
Let us choose finitely many balls $B_1,\ldots,B_m$ and $\phi_{j}\in C^{\infty}_{0}(B_{j})$ for $j=0,1,\ldots,m$ (we take $B_{0}=\D$ here) such that 
\begin{itemize}
\item[i.] $\sum_{j=0}^{m}\phi_{j}=\psi$ on $\Dc,$
 \item[ii.] $\D\cap B_j$ is a domain for  $j=1,2,\ldots,m,$
\item[iii.] $\cup_{j=1}^mB_{j}$ covers the closure of the set $\{z\in b\D: \psi(z) \neq 0\},$
\item[iv.] $\D\cap B_j$ has a compact $\dbar$-Neumann operator for $j=1,2,\ldots,m.$
\end{itemize}
We note that multiplication with smooth functions preserves  the domain of $\dbar^{*}$ and the $\dbar$-Neumann operator is compact on $B_{j}\cap \D$ for $j=1,\ldots, m.$ Compactness of $N$ implies the so-called compactness estimates (see for example \cite{FuStraube01}). Let $W^{-1}(\D)$ denote the Sobolev -1 norm for functions and forms. Then for every $\varepsilon>0$ there exists $ C_{\varepsilon}>0$ such that for $h\in L^{2}_{(0,1)}(\D)$ in the domains of $\dbar$ and $\dbar^{*}$ we have
\begin{eqnarray*}
 \|\psi h \|_{L^{2}(\D)} &\leq &\sum_{j=0}^{m} \|\phi_{j}h\|_{L^{2}(\D)} \\  
&\lesssim&  \sum_{j=0}^{m} \varepsilon\Big( \|\dbar (\phi_{j}h)\|_{L^{2}(\D)}+ \|\dbar^{*} (\phi_{j}h)\|_{L^{2}(\D)} \Big)+ C_{\varepsilon}\|\phi_{j}h\|_{W^{-1}(\D)}\\ 
&\lesssim& \varepsilon\Big( \|\dbar h\|_{L^{2}(\D)}+ \|\dbar^{*} h \|_{L^{2}(\D)}+\|h\|_{L^{2}(\D)} \Big)+ C_{\varepsilon}\|h\|_{W^{-1}(\D)}.
\end{eqnarray*}   
In the calculations above we used interior ellipticity for $j=0$ and the fact that multiplication by a smooth function is a continuous operator on Sobolev spaces. Now if we replace $h$ by $Nh$ and use the fact that $\|Nh\|_{L^{2}(\D)}+\|\dbar Nh\|_{L^{2}(\D)}+\|\dbar^{*}Nh\|_{L^{2}(\D)} \lesssim \|h\|_{L^{2}(\D)}$ we get 
\begin{eqnarray*}
\|\psi Nh \|_{L^{2}(\D)} 
&\lesssim& \varepsilon \|h \|_{L^{2}(\D)}+ C_{\varepsilon}\|Nh\|_{W^{-1}(\D)} \text{ for } h\in L^{2}_{(0,1)}(\D).
\end{eqnarray*}
Then Lemma \ref{CompactEst} implies that $\psi N$ is compact on $L^{2}_{(0,1)}(\D).$ Then using the small constant-large constant inequality $(2ab\leq \epsilon a^2+b^2/\epsilon)$ combined with the inequality above and \eqref{EqnComp1} we get that  for any $\varepsilon >0$ there exists $C_{\varepsilon}>0$ such that 
\begin{equation}\label{Eqn5}
|\langle N(g \dbar  \widetilde{\beta}), g\psi_j \rangle_{L^{2}(\D)}| \leq \varepsilon \|g\|_{L^{2}(\D)}^2+C_{\varepsilon}\|N(g\dbar \widetilde{\beta})\|^2_{W^{-1}(\D)} \text{ for } g\in A^{2}(\D). 
 \end{equation}
Since $\psi_j$ converges to $\dbar \widetilde{\beta}$ uniformly on $\Dc$ for every $\varepsilon>0$ there exists $\psi_j$ such that $|\langle N(g \dbar  \widetilde{\beta}), g(\dbar \widetilde{\beta} -\psi_j)\rangle_{L^{2}(\D)}|\leq \varepsilon \|g\|_{L^{2}(\D)}^2.$ Furthermore, the last inequality together with \eqref{Eqn5} imply that there exists $C_{\varepsilon}>0$ such that  
\[\|\dbar^{*}N(g \dbar \widetilde{\beta})\|_{L^{2}(\D)}^2= \| H^{\D}_{\widetilde{\beta}}(g)\|_{L^{2}(\D)}^2 \lesssim \epsilon \|g\|_{L^{2}(\D)}^2+C_{\epsilon} \|N(g\dbar \widetilde{\beta})\|^2_{W^{-1}(\D)} \text{ for } g\in A^{2}(\D).\]
The above inequality combined with Lemma \ref{CompactEst} and the fact that $W^{-1}(\D)$ imbeds compactly into $L^{2}(\D)$ imply that $H^{\D}_{\widetilde{\beta}}$ is compact on $A^2(\D).$ Therefore, $H^{\D}_{\beta}$ is compact.
\end{proof}

\section{Acknowledgement}
We would like to thank the referee and Emil Straube for helpful comments. 

%%%%%%%%%%%%%%%%%%%%%%
\singlespace

\end{document}